\documentclass{amsart}
\usepackage{amssymb,amsbsy,amsmath,amsfonts,amssymb,amscd,amsthm}
\usepackage[T1]{fontenc}
\usepackage[utf8]{inputenc}
\usepackage{graphicx}
\usepackage[all]{xy}
\usepackage{pdfpages}
\usepackage[linktocpage]{hyperref}

\newcommand{\LF}{\overrightarrow{\Delta}}
\newtheorem{theorem}{Theorem}
\newtheorem{lemma}[theorem]{Lemma}
\newtheorem{proposition}[theorem]{Proposition}

\begin{document}
\title[Littlewood-Paley-Stein functions]{Littlewood-Paley-Stein functions for Hodge-de Rham and Schrödinger operators}
\author{Thomas Cometx}
\email{thomas.cometx@math.u-bordeaux.fr}
\maketitle              
\begin{abstract}
We study the Littlewood-Paley-Stein functions associated with Hodge-de Rham and Schrödinger operators on Riemannian manifolds. Under conditions on the Ricci curvature we prove their  boundedness on $L^p$ for $p$ in some interval $(p_1,2]$ and make a link to the Riesz Transform. An important fact is that we do not make assumptions of doubling measure or estimates on the heat kernel in this case. For $p > 2$ we give a criterion to obtain the boundedness of the vertical Littlewood-Paley-Stein function associated with Schrödinger operators on $L^p$.
\end{abstract}
\tableofcontents
\section{Introduction and main results}

Let $(M,g)$ be a non-compact Riemannian Manifold of dimension $n$. Here $g$ denotes the Riemannian metric that gives a smooth inner product $g_x$ on each tangent space $T_x M$. It induces a smooth inner product on cotangent spaces $T^*_x M$ which we denote by $<,>_x$ and the Riemannian measure $dx$. We note $|.|_x = <.,.>_x^{1/2}$ the associated norm on the cotagent space. Let $L^p$ = $L^p(M, \Lambda^1 T^* M)$ with norm $\| \omega \|_p = \left[ \int_M |\omega(x)|_x^p dx \right]^{\frac{1}{p}}$. For $p=2$, we have the scalar product on $L^2(M, \Lambda^1 T^* M)$ which we denote by $(\alpha,\beta)_{L^2} = \int_M <\alpha,\beta>_x dx$.
Let $\Delta$ be the non-negative Laplace-Beltrami on functions, $\overrightarrow{\Delta}$ be the Hodge-de Rham Laplacian on 1-differential forms. It is defined by $\overrightarrow{\Delta} =dd^* + d^*d $ where $d$ is the exterior derivative and $d^*$ its adjoint for the  $L^2$ scalar product. Let $\tilde{\Delta} = \nabla^* \nabla$ be the rough Laplacian on forms, where $ \nabla$ is the Levi-Civita connection. A link can be done between the previous two operators via the Ricci curvature. Indeed, let $R$ be the Ricci tensor on 1-forms, then the Böchner formula says

\begin{equation}\label{bochner}
 \overrightarrow{\Delta} w   = \tilde{\Delta}  w + R w.
\end{equation}

The present work is devoted to the study of $L^p$ boundedness of the Littlewood-Paley Stein functions associated with the Hodge-de Rham Laplacian on 1-forms as well as Schrödinger operators on functions. We also study the $L^p$ boundedness of the Riesz transform $\mathcal{R} = d \Delta^{-1/2}$. In contrast to previous works on Riesz transforms we do not assume the doubling volume property or a Gaussian estimate for the heat kernel.

The vertical Littlewood-Paley-Stein function for the Laplace-Beltrami operator on functions  was introduced by Stein and is defined by

\begin{align*}
    G(f)(x) &:= \left[\int_0^\infty |d e^{-t\sqrt{\Delta}} f(x)|^2 \,t  dt\right]^{1/2}.
\end{align*} 

The horizontal one is defined by 

\begin{align*}
    g(f)(x) &: = \left[\int_0^\infty |\frac{\partial}{\partial t} e^{-t\sqrt{\Delta}} f(x)|^2\,t dt\right]^{1/2}.
\end{align*} 

The functionals $G$ and $g$ are always, up to a multiplicative constant, isometries of $L^2(M)$. An interesting question is to find the range of $p$ such that $g$ and $G$ extend to bounded operators on $L^p(M)$. In \cite{steinf}, Stein proved that they are bounded on $L^p$ in the euclidean setting for all $p \in (1,\infty)$ and of weak type (1,1). In \cite{steinb}, he proved the boundedness of $G$ for $1 < p < \infty$ in the case where M is a compact Lie group and for $p \in (1,2]$ without any assumption ont he manifold. He also proved the boundedness of $g$ for a general Markov semigroup. The boundedness of the horizontal Littlewood-Paley-Stein function is related to the existence of $H^\infty$ functional calculus for the generator of the semigroup (see Cowling et al. \cite{mcint}).
 Coulhon, Duong and Li proved in \cite{lps} that if the heat kernel admits a Gaussian upper estimate and the manifold satisfies the volume doubling property, then $G$ is of weak type $(1,1)$. In \cite{lohoue}, Lohoué treated the case of Cartan-Hadamard manifolds. In \cite{Meyer}, \cite{Meyer2} Paul-André Meyer studied these functionals with probabilistic methods.

The boundedness of $G$ is linked with the boundedness of the Riesz transform $\mathcal{R} = d \Delta^{-1/2}$. Riesz transform always extends to a bounded operator from $L^2(M)$ to $L^2(M, \Lambda^1 T^* M)$ and it is a major question in harmonic analysis to find the range of $p$ for which it extends to a bounded operator on $L^p$. It is the case for $p\in (1,2)$ under the assumptions of volume doubling property and Gaussian upper estimate \cite{riesz2}. For $p>2$  the situation is complicate. The Riesz transform is bounded on $L^p$ for all $p \in (1, \infty)$ if the Ricci curvature is non-negative by a well known result by Bakry \cite{bakry}. A counter example for large $p$ is given in \cite{riesz2} for a manifold satisfying the volume doubling property and the Gaussian bound. We refer to Carron-Coulhon-Hassell \cite{carron} for precise results on such manifolds. A sufficient condition for the boundedness of the Riesz transform in terms of the negative part of the Ricci curvature is given by Chen-Magniez-Ouhabaz \cite{chenmagouh}.

For manifolds either without the volume doubling property or the Gaussian bound little is known. It is an open problem wether the Riesz transform is always bounded on $L^p$ for $p \in (1,2)$.

In this article, we study the boundedness of the vertical Littlewood-Paley-Stein functions associated with the Hodge-de Rham operator defined as follows

\begin{align*}
    G_{\overrightarrow{\Delta}}^+(\omega)(x) &:= \left[\int_0^\infty |d e^{-t\sqrt{\LF}} \omega|_x^2\, tdt)\right]^{1/2}, \\
     G_{\overrightarrow{\Delta}}^-(\omega)(x) &:= \left[\int_0^\infty |d^* e^{-t\sqrt{\LF}} \omega(x)|^2\, t dt)\right]^{1/2}, \\
     G_{\overrightarrow{\Delta}}(\omega)(x) &:= \left[\int_0^\infty |\nabla e^{-t\sqrt{\LF}} \omega|_x^2\, t dt)\right]^{1/2}.
\end{align*}

Note that $d e^{-t\sqrt{\LF}} \omega$ is a 2-differential form and $|d e^{-t\sqrt{\LF}} \omega|_x$ is defined as before. We also define the functional 

\begin{equation*}
    H_{\overrightarrow{\Delta}} (\omega)(x) := \left( \int_0^\infty |\nabla e^{-t\LF} \omega|_x^2 + <(R^+ + R^-) e^{-t\LF} \omega, e^{-t\LF} \omega>_x dt \right).
\end{equation*}
where $R^+$ and $R^-$ respectively are the positive and negative part of the Ricci curvature. 

One can also define the horizontal functions for these operators by $$\overrightarrow{g}(\omega)(x) = \left[\int_0^\infty |\frac{\partial}{\partial t} e^{-t\sqrt{\LF}} \omega|_x^2 t dt\right]^{1/2}.$$

Our main contribution is the $L^p$ boundedness of the vertical functional $H_{\overrightarrow{\Delta}}$ associated with the Hodge Laplacian. We work outside the usual setting, that is without assuming the manifold has the volume doubling property, or  its heat kernel satisfies a Gaussian estimate. Instead, we rely on two other hypothesis : a maximal inequality for the semigroup and subcriticality of the negative part of the Ricci curvature. Under these properties we prove the boundedness of all the previous vertical Littlewood-Paley-Stein functions. More precisely

\begin{theorem}\label{theoH1}
Suppose that the negative part $R^-$ of the Ricci curvature is subcritical with rate $\alpha \in (0,1)$ that is, for all $\omega \in D(\LF)$
    \begin{equation}\label{soucr}
       (R^- \omega, \omega)_{L^2} \leq \alpha ((\tilde{\Delta} + R^+) \omega, \omega)_{L^2}.
    \end{equation}
Let $p_1 = \frac{2}{1+\sqrt{1-\alpha}}$.  Given $ p \in (p_1,2]$ and assume that $e^{-t\LF}$ satisfies the maximal inequality
    
    \begin{equation} \label{maxineq}
        \| \sup_{t>0} |e^{-t\LF} \omega |_x \|_p \leq C \|\omega\|_p.
    \end{equation} Then $H_{\overrightarrow{\Delta}}$, $G_{\overrightarrow{\Delta}}^+$, $G_{\overrightarrow{\Delta}}^-$ and $G_{\overrightarrow{\Delta}}$ are bounded on $L^p$. They are also bounded on $L^q$ for all $q \in (p,2]$.
\end{theorem}

Note that we can write Theorem 1 restricted to exact 1-forms assuming \eqref{maxineq} only on these forms.

As a consequence we obtain the following result on the Riesz Transform.

\begin{theorem}\label{theoRiesz}
 Suppose that the negative part of the Ricci curvature $R^-$ satisfies \eqref{soucr} for some $\alpha \in (0,1)$ and let $p_1 = \frac{2}{1+\sqrt{1-\alpha}}$. Given $p$ in $(p_1,2]$ and suppose that $e^{-t\LF}$ satisfies the maximal inequality \eqref{maxineq}. Then the Riesz transform is bounded on $L^{p'}$ where $\frac{1}{p} + \frac{1}{p'} = 1$. It is also bounded on $L^q$ for all $q \in [2,p')$.
\end{theorem}

Magniez, in \cite{mag}, proved the boundedness of $\mathcal{R}$  on $L^p$ for $p$ in a slightly bigger interval by assuming the doubling property and the Gaussian upper estimate for the heat kernel. As mentioned above, we do not assume any of these two properties. Instead we assume the maximal inequality (\ref{maxineq}). Note that if the heat semigroup on functions satisfies the so-called  gradient estimate \begin{equation}\label{gradientestimate}
|de^{-t\Delta} f|_x \leq M e^{-\delta t \Delta} |d f|_x
\end{equation}with some postive constants $M$ and $\delta$ then \eqref{maxineq} is satisfied on exact forms. Indeed, in this case, $|e^{-t\overrightarrow{\Delta}}df|_x = |d e^{-t\Delta} f|_x\leq M  e^{-\delta t{\Delta}} |df|_x \leq M \sup_{t>0} e^{-\delta t\Delta} |df|_x$ where the right hand side term is bounded on $L^p$ for $p \in(1,\infty)$ because $e^{-t\Delta}$ is submarkovian (see \cite{steinb}, p 73). If in addition one has $L^p$-decomposition on forms $\omega = df + d^* \beta$ with $\|\omega\|_p \simeq \|d f\|_p$ then \eqref{gradientestimate} implies \eqref{maxineq}. The latter decomposition is not true on all non-compact Riemannian manifolds. See \cite{riesz} for a discussion on this property. If $R \geq 0$ then obviously \eqref{maxineq} and \eqref{gradientestimate} are satisfied since $|e^{-t\LF} \omega|_x \leq e^{-t\Delta} | \omega|_x$ as a consequence of the Böchner formula \eqref{bochner}.

We also study the boundedness of the Littlewood-Paley-Stein function associated with the Schrödinger operators $L = \Delta + V$, namely

\begin{align*}
   G_L(f)(x) &:= \left[\int_0^\infty \left( |d e^{-tL^{1/2}} f|_x^2 + |V|(e^{-tL^{1/2}}f)^2(x) \right) t dt\right]^{1/2}, \\
   H_L(f)(x) &:= \left[\int_0^\infty \left( |d e^{-tL} f|_x^2 + |V|(e^{-tL}f)^2(x) \right)  dt\right]^{1/2}. \\
\end{align*}
We use the classical notation $V^+$ and $V^-$ for the positive and negative parts of $V$. We take $V^+ \in L^1_{loc}(M)$. The Schrödinger operator $L$ is defined by the quadratic form techniques.

\begin{theorem}\label{theoHL}
Let $L = \Delta + V$ be a Schrödinger operator such that the negative part $V^-$ is subcritical with rate $\alpha \in (0,1)$ in the sense
\begin{equation}\label{48}\int_M V^- f^2 dx \leq \alpha \int_M \left( V^+ f^2 + |\nabla f|^2 \right) dx, \quad \forall f \in C_c^\infty(M) .\end{equation}
Then $H_L$ and $G_L$ are bounded on $L^p(M)$ for all $p \in (p_1, 2]$ where $p_1 = \frac{2}{1+\sqrt{1-\alpha}}.$ 
\end{theorem} The result was known for non-negative $V$ (see Ouhabaz \cite{ouh}). In this case $\alpha = 0$ and then $H_L$ and $G_L$ are bounded on $L^p$ for $p \in (1,2]$.  
The paper is organized as follows. In section 2, we recall links between Riesz transform and Littlewood-Paley-Stein functions. In section 3 we prove Theorems 1 and 2. In section 4, using the same techniques, we give a short proof of a result by Bakry in \cite{bakry} saying that the modified Riesz transform $d(\Delta + \epsilon)^{-1/2}$ is bounded for $p > 2$ if we suppose the Ricci curvature is bounded from below. In section 5, using the same techniques again, we study the boundedess of the horizontal LPS function associated with $\LF$. In section 6 we prove Theorem \ref{theoHL}. In section 7, we assume the doubling property and the Gaussian estimate for heat kernel to give a criterion on $V$ to obtain the boundedness of $H_L$ in the case $p>2$. 

\textbf{Notations :} Let $p>1$. During all the paper, we denote by $L^p$ either the spaces $L^p(M)$ and $L^p(M, \Lambda^1 T^*M)$ when the context is clear. We sometimes denote by $L^p(\Lambda^1 T^* M)$ the space $L^p(M, \Lambda^1 T^*M)$. We denote by $p'$ the conjugate expontent of $p$ defined by $\frac{1}{p} + \frac{1}{p'} =1$. We denote by $C_c^\infty(M)$ the space of smooth compactly supported functions on $M$. We often use $C$ and $C'$ for possibly different positive constants.

\section{The Littlewood-Paley-Stein fuctions and the Riesz Transform}

The aim of this section is to recall the links between Littlewood-Paley-Stein functions, Riesz transforms and other estimates. We show a duality argument which shows why the function $G^-_{\overrightarrow{\Delta}}$ is useful to study the Riesz transform.

First, we have the following theorem which is taken from \cite{riesz}.

\begin{theorem}\label{theoCoulh}
Let $p \in (1, \infty)$. If $G$ is bounded on $L^p(M)$ and $ \overrightarrow{g}$ is bounded from $L^{p'} (\Lambda^1 T^*M)$ to $L^{p'}(M)$, then the Riesz transform extends to a bounded operator from $L^p(M)$ to $L^p(\Lambda^1 T^* M)$.  
\end{theorem} 

The subordination formula for positive and self-adjoint operators

\begin{equation}\label{sub}
    e^{-tA^{\frac{1}{2}}} = \frac{1}{\sqrt{\pi}} \int_0^\infty e^{-\frac{t^2}{4s} A} e^{-s} \frac{ds}{\sqrt{s}}
\end{equation}
gives the following pointwise domination (see e.g. \cite{lps}).

\begin{proposition}\label{propsubor} For all $f \in C^\infty_c(M)$, for all $\omega \in C^\infty_c(\Lambda^1 T^*M)$, for all $x \in M$,
\begin{equation}\label{pointy}
  \left\{
      \begin{aligned} 
     G(f)(x) &\leq C H(f)(x) \\
     G_{\overrightarrow{\Delta}} (\omega) (x) &\leq C' H^{(\nabla)}_{\overrightarrow{\Delta}}(\omega)(x) \leq C' H_{\overrightarrow{\Delta}} (\omega)(x)
\end{aligned} \right.
\end{equation}
  where $C$ and $C'$ are positive constants and $H$ and $H^{(\nabla)}_{\overrightarrow{\Delta}}$ are defined by
 \begin{align*}
    H(f)(x) &= \left[ \int_0^\infty |d e^{-t\Delta} f|_x^2 dt \right]^{1/2} \\
    H_{\overrightarrow{\Delta}}^{(\nabla)}(\omega)(x)&=\left[ \int_0^\infty |\nabla e^{-t\LF} \omega|_x^2 dt \right]^{1/2} .
 \end{align*}
\end{proposition}

In order to study the boundedness of the Riesz transform $\mathcal{R} = d \Delta^{-1/2}$ on $L^p$, we argue by duality. It is sufficent to prove the boundedness of the adjoint $\mathcal{R}^* = \Delta^{-1/2} d^*$ on $L^{p'}$ to obtain the boundedness of the Riesz transform on $L^p$. By the classical commutation property, $\mathcal{R}^* = d^* \overrightarrow{{\Delta}}^{-1/2}.$ Therefore we consider $d^* \overrightarrow{\Delta}^{-1/2}$ on $L^{p'}$.  

In the next result, we have a version of Theorem \ref{theoCoulh} in which we appeal to $G_{\overrightarrow{\Delta}}^-$ instead of $\overrightarrow{g}$ and $G$. 

\begin{theorem}\label{theoRieszLPS}
Let $p \in (1, \infty)$. If $ G_{\overrightarrow{\Delta}}^-$ is bounded from $L^{p'} (\Lambda^1 T^*M)$ to $L^{p'}(M)$, then the Riesz transform extends to a bounded operator on $L^p$ (with values in $L^p(\Lambda^1 T^*M)$). 
\end{theorem} 

\begin{proof}
We show that ${d^*} {\overrightarrow{\Delta}}^{-1/2}$ is bounded from $L^{p'}(\Lambda^1 T^* M)$ to $L^{p'}(M)$. The proof is the same as for Theorem \ref{theoCoulh}. We write the argument for the sake of completeness. Let $ \omega \in L^{p'}(\Lambda^1 T^* M)$. We have by duality 

\begin{align*}
\| d^* \overrightarrow{\Delta}^{-1/2} \omega \|_{p'} &= \sup_{g \in L^p(M), \|g\|_p = 1} \left| \int_M d^* \overrightarrow{\Delta}^{-1/2} \omega (x) g(x) dx \right|.
\end{align*}
By the reproducing formula $$\int_M f(x) g(x) \,dx =  4 \int_M  \int_0^{\infty} [\frac{\partial}{\partial t} e^{-\sqrt{\Delta}} f(x)] [\frac{\partial}{\partial t} e^{-\sqrt{\Delta}} g(x)] \, t dt dx$$ applied with $f = d^* \LF^{-1/2} \omega$ we have\begin{equation*}
    \int_M d^* \overrightarrow{\Delta}^{-1/2} \omega (x) g(x) dx = 4 \int_0^\infty \int_M [\frac{\partial}{\partial t} e^{-t \sqrt{\Delta} } d^* \overrightarrow{\Delta}^{-1/2} \omega (x)] [\frac{\partial}{\partial t}e^{-t \sqrt{\Delta} }g(x)] t dx  dt .
\end{equation*}
Using the commutation formula $d^* \LF = \Delta d^*$, we have \begin{align*}
    \frac{\partial}{\partial t} e^{-t \sqrt{\Delta} } d^* \overrightarrow{\Delta}^{-1/2} \omega  &= -\sqrt{\Delta} e^{-t \sqrt{\Delta} } d^* \overrightarrow{\Delta}^{-1/2} \omega \\
    &=  - d^* \sqrt{\LF}e^{-t \sqrt{\LF} } \overrightarrow{\Delta}^{-1/2} \omega  \\
    &= -d^* e^{-t\sqrt{\LF}} \omega.
\end{align*}
Thus, 
\begin{align*}
\| d^* \overrightarrow{\Delta}^{-1/2} \omega \|_{p'} &=4 \sup_{g \in L^p(M), \|g\|_p = 1} \left| \int_0^{\infty} \int_M  [d^* e^{-t\sqrt{\LF}} \omega (x)]  [\frac{\partial}{\partial t} e^{-t\sqrt{\Delta}}g(x)]t dx dt \right| \\
& \leq 4\sup_{g \in L^p(M), \|g\|_p = 1} \left\| \left(\int_0^\infty |d^* e^{-t\sqrt{\LF}} \omega|_x^2 t dt \right)^{1/2} \right\|_{p'}  \\ 
& \hspace{3cm} \times \left\|  \left(\int_0^\infty |\frac{\partial}{\partial t} e^{-t\sqrt{\Delta}}g(x)|^2 t dt \right)^{1/2} \right\|_p \\
&=4 \sup_{g \in L^p(M), \|g\|_p = 1} \|G_{\overrightarrow{\Delta}}^-(\omega) \|_{p'} \|g(f) \|_{p}\\ 
&\leq C \| \omega \|_{p'}. 
\end{align*}
Here we used the boundedness of $G_{\overrightarrow{\Delta}}^-$ on $L^{p'}$ which is our assumption. Note that $g$ is bounded on $L^p$ for all $p \in (1, \infty)$ by \cite{steinb}, p111.
\end{proof}
\section[Vertical LPS functions associated with Hodge-de Rham Laplacian for $p \leq 2$]{Vertical LPS functions for the Hodge-de Rham Laplacian for $p \leq 2$}

In this section we prove Theorem 1 and 2. We start with some useful pointwise inequalities on smooth differential forms. 

\begin{lemma}\label{lemmaIneq}
For $p\leq2$, for all suitable differential form $\omega$ and for all $x \in M$ we have the pointwise inequality
\begin{equation}\label{ineq1}
    - \Delta |\omega|_x^p \geq |\omega|_x^{p-2} \left[ p (p-1) |\nabla \omega|_x^2 - p <\tilde{\Delta} \omega,\omega >_x \right].
\end{equation}
\end{lemma}

\begin{proof}

We compute the calculations for $p=2$. Let $x \in M$, let $X_i$ be orthonormal coordinates at $x$, and $\theta_i$ their dual basis in the cotangent space, satisfying $\nabla \theta_i = 0$ for all $i$ at $x$. We have \begin{equation*}
    \tilde{\Delta} \omega = - \sum_{i=1}^n \nabla_{X_i} \nabla_{X_i} \omega.
\end{equation*}
Hence,
\begin{align*}
    \tilde{\Delta} |\omega|_x^2 &=  -\sum_{i=1}^n \nabla_{X_i} \nabla_{X_i} <\omega,\omega>_x \\
    &=  - 2\sum_{i=1}^n\nabla_{X_i} <\nabla_{X_i} \omega,\omega>_x \\
    &=- 2 \sum_{i=1}^n \left[  <\nabla_{X_i} \nabla_{X_i} \omega, \omega>_x + |\nabla_{X_i} \omega|_x^2 \right]\\
    &= 2 <\tilde{\Delta} \omega , \omega>_x - 2 \sum_{i=1}^n |\nabla_{X_i} \omega|_x^2 \\
    &= 2 <\tilde{\Delta} \omega, \omega>_x - 2 |\nabla \omega|_x^2.
\end{align*}
In order to obtain \eqref{ineq1} for $p <2$, we recall that for all suitable functions $f$ we have
\begin{equation}\label{17}
    - \Delta f^\frac{p}{2} = \frac{p(p-2)}{4} |\nabla f|^2 f^{\frac{p-4}{2}} - \frac{p}{2}f^\frac{p-1}{2} \Delta f.
\end{equation}
We apply \eqref{17} with $f=|\omega|^2$ and the equality we proved for $p=2$ to obtain
\begin{align*}
        - \Delta |\omega|_x^p &=  \frac{p(p-2)}{4} \left|(\nabla |\omega|_x^2)\right|_x^2|\omega|_x^{p-4}- \frac{p}{2}|\omega|_x^{p-2} \Delta |\omega|_x^2 \\
        &=\frac{p(p-2)}{4} |(\nabla |\omega|_x^2)|_x^2|\omega|_x^{p-4} - p |\omega|_x^{p-2} < \tilde{\Delta} \omega, \omega>_x + p |\omega|_x^{p-2} |\nabla \omega|_x^2.
\end{align*}
Consequently, it is sufficient to show that 
\begin{equation}\label{18}
    \frac{p(p-2)}{4} \left|(\nabla  |\omega|_x^2)\right|_x^2|\omega|_x^{p-4}  \geq p(p-2) |\omega|_x^{p-2} |\nabla \omega|_x^2.
\end{equation}
Since $p < 2$, \eqref{18} is equivalent to 
\begin{equation}\label{19}
   \left|\nabla  (|\omega|_x^2)\right|_x^2 \leq 4 |\omega|_x^2 |\nabla \omega|_x^2.
\end{equation}
We prove \eqref{19} using local coordinates. We write $\omega = \sum_{j=1}^n f_j \theta_j$, so that $|\omega|_x^2 = \sum_{i=1}^n f_i^2$ and
\begin{align*}
      \left| (\nabla  |\omega|_x^2)\right|_x^2 &= | \sum_{i=1}^n 2 f_i d f_i |_x^2 \\
      &= 4 \sum_{j=1}^n \sum_{i=1}^n f_i f_j <d f_i, d f_j>_x \\
      &\leq 4 \sum_{j=1}^n f_i^2 \sum_{i=1}^n |d f_j|^2_x  \\
      &= 4 |\omega|_x^2 |\nabla \omega|_x^2
\end{align*}
where we used Cauchy-Schwarz inequality in $\mathbb{R}^n$.
\end{proof}
We will also need other the following inequalities from \cite{mag}.

\begin{lemma}\label{lemmaMagn}
For any $p>1$ and any suitable $\omega$, we have the pointwise inequality
\begin{equation}\label{ineq2}
    |\nabla \left( |\omega|_x^{\frac{p}{2}-1} \omega \right) |_x^2 \leq \frac{p^2}{4(p-1)} <\nabla \left( |\omega|_x^{p-2}\omega \right),\nabla \omega>_x.
\end{equation}
\end{lemma}

\begin{lemma}\label{lemmaMagn2}
For any suitable 1-differential form $\omega$ we have the pointwise estimates $|d\omega|_x \leq 2 |\nabla \omega|_x$ and $|d^* \omega|(x)\leq \sqrt{n} |\nabla \omega|_x$.
\end{lemma}

We recall that the negative part of the Ricci curvature is sub-critical with rate $\alpha \in (0,1)$ if for all suitable $\omega$ we have \eqref{soucr}

\begin{equation*} 
   (R^- \omega,\omega)_{L^2} \leq \alpha \left[  (R^+ \omega,\omega)_{L^2} + \|\nabla \omega \|_2^2 \right].
\end{equation*}
Note that \eqref{soucr} is equivalent to
\begin{equation*}
       (\LF \omega,\omega)_{L^2} \geq (1-\alpha) ((\tilde{\Delta} + R^+) \omega,\omega)_{L^2}.
\end{equation*}
We have the following analogous inequality on $L^p$.

\begin{proposition}\label{propSubc}
 If the negative part of the Ricci curvature is subcritical with rate $\alpha$, then for all suitable non vanishing $\omega$ in $L^p$
 \begin{equation}\label{SCP}
 \int_M <R^-\omega  ,  \omega >_x |\omega|_x^{p-2} dx \leq \alpha  \int_M <R^+ \omega  ,  \omega >_x |\omega|_x^{p-2}  + |\nabla \left( |\omega|_x^{\frac{p}{2}-1} \omega \right) |_x^2 dx.
 \end{equation}
\end{proposition}

\begin{proof}
Let $\omega \in L^p(\Lambda^1 T^*M)$ a suitable differential form and let $ \beta = |\omega|_x^{\frac{p}{2}-1} \omega.$ We have $\beta \in L^2(\Lambda^1 T^*M)$ and then we apply \eqref{soucr} to $\beta$ to obtain

\begin{multline*}
    \int_M <R^-|\omega|_x^{\frac{p}{2}-1} \omega  , |\omega|_x^{\frac{p}{2}-1} \omega >_x dx \leq  \alpha  \int_M <R^+|\omega|_x^{\frac{p}{2}-1} \omega  , |\omega|_x^{\frac{p}{2}-1} \omega >_x \\ + |\nabla \left( |\omega|_x^{\frac{p}{2}-1} \omega \right) |_x^2 dx.
    \end{multline*}
   Using the pointwise linearity of $R^+$ and $R^-$ it leads to
    \begin{align*}
    \int_M <R^-\omega  ,  \omega >_x |\omega|_x^{p-2} dx &\leq \alpha  \int_M <R^+ \omega  ,  \omega >_x |\omega|_x^{p-2}  + | \nabla \left( |\omega|_x^{\frac{p}{2}-1} \omega \right) |_x^2 dx  .
\end{align*} which is the desired inequality.\end{proof}

The following proposition is proven in \cite{mag}. We reproduce the proof for the sake of completeness. 
\begin{proposition}\label{propDecreasing}
 Suppose that the negative part of the Ricci curvature is subcritical with rate $\alpha \in (0,1)$. Let $p_1 = \frac{2}{1+\sqrt{1-\alpha}}$. Then the norm $\|e^{-t \LF} \omega \|_p $ is a decreasing function of $t$ for all $p \in (p_1,p_1')$. Consequently $e^{-t \LF}$ is a contraction semigroup on $L^p$ for all $p \in (p_1,p_1')$.
\end{proposition}

\begin{proof}

Let $\omega$ be a suitable smooth differential 1-form and $\omega_t = e^{-t\LF} \omega$.  We compute the derivative of $E(t) = \|\omega_t\|_p^p$. We have

\begin{align*} 
    \frac{\partial E(t)}{\partial t} &= \int_M \frac{\partial}{\partial t} |\omega_t|_x^p dx \\
    &= - p \int_M <\LF \omega_t,  |\omega_t|_x^{p-2} \omega_t>_x dx \\
    &= - p \int_M <\nabla \omega_t, \nabla |\omega_t|_x^{p-2} \omega_t>_x + <(R^+ - R^-) \omega_t,   |\omega_t|_x^{p-2}\omega_t>_x dx \\
    &= - p \int_M <\nabla \omega_t, \nabla |\omega_t|_x^{p-2} \omega_t>_x + <(R^+ - R^-) |\omega_t|_x^{\frac{p}{2} - 1} \omega_t ,   |\omega_t|_x^{\frac{p}{2} - 1}\omega_t>_x dx \\
    &\leq - p \int_M   \frac{4(p-1)}{p^2}|\nabla |\omega_t|_x^{\frac{p}{2} - 1} \omega_t|_x^2 + <(R^+ - R^-) |\omega_t|_x^{\frac{p}{2} - 1} \omega_t ,   |\omega_t|_x^{\frac{p}{2} - 1}\omega_t>_x dx.
\end{align*}
We used Lemma \ref{lemmaMagn} to obtain the last inequality. By the subcriticality assumption we have
\begin{multline}
 \frac{\partial E(t)}{\partial t} = - p  (1- \alpha) \int_M <R^+  |\omega_t|_x^{\frac{p}{2} - 1} \omega_t ,   |\omega_t|_x^{\frac{p}{2} - 1}\omega_t>_x dx \\
 - p \left( \frac{4(p-1)}{p^2} - \alpha \right)  \int_M \left| \nabla|\omega_t|_x^{\frac{p}{2} - 1} \omega_t \right|_x^2 dx.
\end{multline} 
Hence $\frac{\partial E(t)}{\partial_t} \leq 0$ for $p$ such that $4(p-1) \geq \alpha p^2$. This is equivalent to $ p \in (p_1, p_1')$ where $p_1 = \frac{2}{1+\sqrt{1-\alpha}}$. \end{proof}

This result ensures the existence of $e^{-t \LF} \omega$ in $L^p(\Lambda^1 T^*M)$ and then one can consider Littlewood-Paley-Stein functions associated with $\LF$ on $L^p$. \newline

\textbf{Proof of Theorem 1.} By Lemma \ref{lemmaMagn2} and Proposition \ref{propsubor} it is sufficient to prove the boundedness of $H_{\LF}$. We follow similar arguments as in \cite{steinb}, p52-54. Let $\omega$ be a smooth non vasnishing 1-differential form and $\omega_t := e^{-t\LF} \omega$. A direct calculation and Lemma \ref{lemmaIneq} give

\begin{equation}\label{145}
  \left\{
      \begin{aligned}
   -  \Delta |\omega|_x^p &\geq  -p <\tilde{\Delta} \omega, \omega>_x |\omega|_x^{p-2} + p(p-1)|\omega|_x^{p-2}|\nabla \omega|_x^2  \\
    - \frac{\partial}{\partial t} |\omega|_x^p &= p <\LF \omega, \omega>_x |\omega|_x^{p-2}.
\end{aligned} \right.
\end{equation}
Using the Böchner formula \eqref{bochner} we obtain
\begin{equation*}
   - \frac{\partial}{\partial t} |\omega_t|_x^p =  p <(\tilde{\Delta} + R^+ -R^-) \omega_t, \omega_t>_x |\omega_t|_x^{p-2}. 
\end{equation*}
Let $\xi, c, k$ be positive constants and set\begin{equation*}
 Q(\omega,x,t) := - \frac{\partial}{\partial t} |\omega_t|_x^p- \xi \Delta |\omega_t|_x^p + < (-c R^+ + k R^-) \omega_t, \omega_t>_x |\omega_t|_x^{p-2}.
 \end{equation*}
Using \eqref{145} we obtain
\begin{multline*}
    Q(\omega,x,t) \geq |\omega_t|_x^{p-2} \biggl[ p (1- \xi) <\tilde{\Delta} \omega_t, \omega_t>_x  + (p-c) <R^+ \omega_t, \omega_t>_x \\ 
    +  (-p+k)<R^- \omega_t , \omega_t>_x + \xi p(p-1)  |\nabla \omega_t|_x^2 \biggr].
\end{multline*}
Hence, \begin{multline} \label{101}
|\omega_t|_x^{2-p} \left[ Q(\omega,x,t) +   |\omega_t|_x^{p-2} p (\xi-1) <\tilde{\Delta} \omega_t, \omega_t>_x \right] \geq \\
  (p-c)<R^+ \omega_t, \omega_t>_x   +  (-p+k)<R^- \omega_t , \omega_t>_x + \xi p(p-1)  |\nabla \omega_t|_x^2.
  \end{multline}
If the quantities $\epsilon := p-c$ and $\eta: = -p+k$ are positive we obtain
\begin{multline} \label{101bis} <R^+ \omega_t, \omega_t>_x   +  <R^- \omega_t , \omega_t>_x +  |\nabla \omega_t|_x^2  \\ \leq C |\omega_t|_x^{2-p} \left[ Q(\omega,x,t) +   |\omega_t|_x^{p-2} p (\xi-1) <\tilde{\Delta} \omega_t, \omega_t>_x \right]
\end{multline}  for some positive constant $C$ depending on $p, \xi, \epsilon$ and $\eta$. In particular we have the pointwise inequality

\begin{equation}\label{60}
    Q(\omega,x,t) +   |\omega_t|_x^{p-2} p (\xi-1) <\tilde{\Delta} \omega_t, \omega_t>_x \geq 0.
\end{equation}
By integration of \eqref{101bis} for $t \in [0, \infty)$ we obtain
\begin{equation*}
    H_{\overrightarrow{\Delta}}(\omega)(x)^2 \leq C \int_0^\infty |\omega_t|_x^{2-p} \left[ Q(\omega,x,t) +   |\omega_t|_x^{p-2} p (\xi-1) <\tilde{\Delta} \omega_t, \omega_t>_x \right]  dt.
\end{equation*}
Integrating over $M$ yields
\begin{multline*}
    \int_M H_{\overrightarrow{\Delta}}(\omega)(x)^p dx  \leq \\ C  \int_M \left(\int_0^\infty |\omega_t|_x^{2-p} \left[ Q(\omega,x,t) +   |\omega_t|_x^{p-2} p (\xi-1) <\tilde{\Delta} \omega_t, \omega_t>_x \right] dt \right)^{\frac{p}{2}} dx.
\end{multline*}
Hence, by \eqref{60}
\begin{multline*}
    \int_M H_{\overrightarrow{\Delta}}(\omega)(x)^p dx  \leq   \\ C \int_M |\omega^*|^{p(1-\frac{p}{2})} \left[\int_0^\infty Q(\omega,x,t) +   |\omega_t|_x^{p-2} p (\xi-1) <\tilde{\Delta} \omega_t, \omega_t>_x dt \right]^{\frac{p}{2}} dx.
\end{multline*}
We use Hölder's inequality on $M$ with powers $\frac{2}{p}$ and $\frac{2}{2-p}$. By \eqref{60}, we can avoid absolute values in the second integral and obtain
\begin{multline*}
    \int_M H_{\overrightarrow{\Delta}}(\omega)(x)^p dx  \leq C  \left[\int_M |\omega^*|^p\right]^{1-\frac{p}{2}} \times \\ \left[\int_M \int_0^\infty Q(\omega,x,t) +   |\omega_t|_x^{p-2} p (\xi-1) <\tilde{\Delta} \omega_t, \omega_t>_x dt dx \right]^{p/2}
\end{multline*}
where \begin{equation}\label{omegastar}
\omega^*(x) := \sup_{t \geq 0} |\omega_t|_x.\end{equation} The maximal inequality \eqref{maxineq} gives
\begin{multline*}
    \int_M H_{\overrightarrow{\Delta}}(\omega)(x)^p dx  \leq   C  \| \omega \|_p^{p(1-\frac{p}{2})} \times \\ \left[\int_M \int_0^\infty Q(\omega,x,t) +   |\omega|_x^{p-2} p (\xi-1) <\tilde{\Delta} \omega_t, \omega_t>_x dt dx \right]^{p/2} .
\end{multline*}
Set 
\begin{multline*}
    I(t) := \int_M < (k R^- - c R^+)\omega_t,|\omega_t|_x^{p-2}  \omega_t>_x \\
   - p (1-\xi) <\tilde{\Delta} \omega_t,  \omega_t>_x |\omega_t|_x^{p-2}  dx.
\end{multline*}
We have \begin{align*}
&\int_M \int_0^\infty Q(\omega,x,t) + p (\xi -1)|\omega|_x^{p-2}<\tilde{\Delta} \omega_t, \omega_t>_x dt dx \\
&=\int_M \int_0^\infty \left[ -\frac{\partial}{\partial t} |\omega_t|_x^p - \xi \Delta |\omega_t|^p \right] dt dx + \int_0^\infty I(t) dt \\
&= \|\omega\|_p^p + \int_0^\infty I(t) dt \\
&\leq \|\omega\|_p^p
\end{align*}
where we used Lemma \ref{lemmaIsubc} below. Note that we also used that we can choose a sequence $t_n$ tending to $ + \infty$ such that $|\omega_{t_n}|_x$ tends to zero. This is true because $\|\omega_t\|_{2}$ tends to zero. Indeed, let $\omega = \LF \beta$ in the range of $\LF$. One has $\| \omega_t \|_2 = \| \LF e^{-t\LF} \beta \|_2 \leq 
\frac{C}{t} \| \beta \|_2$ by the analyticity of the semigroup on $L^2$. The subcriticality condition implies that there is no harmonic form on $L^2$. Hence the range of $\LF$ is dense and this is still true for all $\omega \in L^2$. Therefore, \begin{equation*}
\int_M H_{\LF}(\omega)(x)^p dx \leq C \| \omega \|_p^{p(1-\frac{p}{2})} \| \omega\|_p^{\frac{p^2}{2}} = C \|\omega\|_p^p.
\end{equation*}
By interpolation we obtain the boundedness on $L^q$ for all $q \in (p,2]$.
\begin{lemma}\label{lemmaIsubc}
Under the subcriticality assumption, for all $p \in (p_1,2]$ we can choose positive constants $c, k$ and $\xi$ such that  $I(t) \leq 0$ for all $t>0$.
\end{lemma}

\begin{proof}

By Proposition \ref{propSubc},
\begin{multline*}
   I(t) \leq  \int_M  (\alpha k - c) <R^+ \omega_t, |\omega_t|_x^{p-2} \omega_t>_x dx + \\
   \int_M \alpha k |\nabla
  \left( |\omega_t|_x^{\frac{p}{2} - 1} \omega_t \right)|_x^{2}  
   -p(1-\xi) <\tilde{\Delta} \omega, |\omega|_x^{p-2} \omega>_x dx .
\end{multline*}
By integration by parts and Lemma \ref{lemmaMagn}, we have
\begin{multline*}
    I(t) \leq  \int_M  (\alpha k - c) <R^+ \omega_t, |\omega_t|_x^{p-2} \omega_t>_x dx  \\ + \int_M \left( \alpha k - \frac{4(1-\xi)(p-1)}{p} \right)|\nabla\left( |\omega_t|_x^{\frac{p}{2} - 1} \omega_t \right)|_x^{2} dx.
\end{multline*}
Choose $c, k$ and $\xi$ such that 
\begin{equation*}
  \left\{
      \begin{aligned}
   p \alpha k&\leq 4 (p-1)(1-\xi) \\
      \alpha k &\leq c.
 \end{aligned} \right.
 \end{equation*}
We can choose $k$ as close to $p$ as we want, so any value of $p$ satisfying $\alpha p^2 - 4(p-1)(1-\xi) < 0$ can be chosen to satisfy the first inequality.
If $\xi < 1 - \alpha$, this inequality is satisfied for all $p$ in the interval
$$\left[ 2 \frac{1-\xi -  \sqrt{(1-\xi)(1-\xi-\alpha)}} {\alpha},2 \frac{1-\xi +  \sqrt{(1-\xi)(1-\xi-\alpha)}} {\alpha}\right]$$
which is contained in $[p_1,p_1']$ with bounds tending to $p_1$ and ${p_1}'$ when $\xi$ tends to zero. Hence, for all $p \in (p_1,2]$ we can choose $k$ and $\eta$ such that $p\alpha k \leq 4(p-1)(1-\xi)$. Since $\alpha \in (0,1)$, given $k \in (p, \frac{p}{\alpha})$ we can choose $c =\alpha k < p$. \end{proof}
\textbf{Proof of Theorem 2.} Let $q \in [2,p')$. By Theorem 1, $G^-_{\LF}$ is bounded on $L^{q'}$. By Theorem 6, the Riesz transform is on bounded on $L^q$.


\section[LPS functions for the Hodge-de Rham Laplacian with Ricci curvature bounded from below]{Vertical LPS functions for the Hodge-de Rham Laplacian with Ricci curvature bounded from below}

We recall Bakry's theorem in \cite{bakry}.
\begin{theorem}\label{theoBakr}
Suppose that the Ricci curvature statisfies $R \geq - \kappa$ with $\kappa > 0$, then the modified Riesz transform $\mathcal{R}_\epsilon = d (\Delta + \epsilon)^{-1/2}$ is bounded from $L^p(M)$ to $L^p(\Lambda^1 T^* M)$, for all $p \in (1,\infty)$ and $\epsilon > 0$.
\end{theorem}

Bakry's paper \cite{bakry} contains several other results some of which are proven by probabilistic methods. An analytic proof of the case $p \in (1,2]$ of the previous theorem is given in \cite{riesz2}. In this section we follow the approach of the previous sections to give a relatively short proof in the case $p \in [2,\infty)$. 

\begin{theorem}\label{Z}
Suppose that the Ricci curvature satisfies $R \geq - \kappa$ with $\kappa \geq 0$, then the functional 

\begin{equation*}
   Z(\omega) := \left[ \int_0^\infty |\nabla e^{-t(\LF + \kappa)} \omega |_x^2 dt \right]^{1/2}
\end{equation*} 
is bounded on $L^p$ for $p \in (1,2]$.
\end{theorem}

\begin{proof}

Set $\omega_t := e^{-t(\LF+ \kappa)} \omega$. Lemma 7 gives
\begin{align*}
    - \Delta |\omega_t|_x^p &\geq p(p-1)|\nabla\omega_t|_x^2 |\omega_t|_x^{p-2} - p < \tilde{\Delta} \omega_t, |\omega_t|_x^{p-2}  \omega_t  >_x \\
    &= p(p-1)|\nabla \omega_t|_x^2 |\omega_t|_x^{p-2} - p < (\LF - R) \omega_t, |\omega_t|_x^{p-2}  \omega_t >_x \\
    &= p(p-1)|\nabla \omega_t|_x^2 |\omega_t|_x^{p-2} - p < (- \frac{\partial}{\partial t} - \kappa - R) \omega_t, |\omega_t|_x^{p-2}  \omega_t >_x \\
    &\geq  p(p-1)|\nabla \omega_t|_x^2 |\omega_t|_x^{p-2} + p <  \frac{\partial}{\partial t} \omega_t, |\omega_t|_x^{p-2}  \omega_t >_x
\end{align*}
where we used $R \geq - \kappa$. Multiplying by $|\omega|_x^{2-p}$ we obtain

\begin{equation}\begin{aligned}
    |\nabla \omega_t|_x^2 &\leq C |\omega_t|_x^{2-p} \left[ - \Delta |\omega_t|_x^p - p  <  \frac{\partial}{\partial t} \omega_t, |\omega_t|_x^{p-2}  \omega_t >_x \right] \\
   \label{pouet} &\leq C |\omega_t|_x^{2-p} \left[ - \Delta |\omega_t|_x^p - \frac{\partial}{\partial t} |\omega_t|_x^p \right].
   \end{aligned}\end{equation}
As a consequence of \eqref{pouet} one has the pointwise inequality
\begin{equation}\label{posi}
- \Delta |\omega_t|_x^p - \frac{\partial}{\partial t} |\omega_t|_x^p \geq  0.
\end{equation}
We integrate \eqref{pouet} for $t \in (0,\infty)$  and use \eqref{posi} to obtain
\begin{align*}
   Z(\omega)(x)^2 &\leq C (\sup_{t>0} |\omega_t|_x)^{2-p} \int_0^{\infty} \left( - \Delta |\omega_t|_x^p - \frac{\partial}{\partial t} |\omega_t|_x^p  \right) dt  \\
   &= C (\sup_{t>0} |\omega_t|_x)^{2-p} \left( |\omega|_x^p - \int_0^\infty \Delta |\omega_t|_x^p dt \right).
\end{align*} 
In the last equality we used the fact that $\lim_{t \rightarrow + \infty} \omega_t = 0$ (in $L^2$). This can be seen from the domination \begin{equation}\label{domination}
|e^{-t(\LF + \kappa)} \omega|_x \leq e^{-t\Delta} |\omega|_x
\end{equation} and $ \lim_{t \rightarrow + \infty} e^{-t\Delta} |\omega|_x = 0$ in $L^2$. The pointwise domination \eqref{domination} is proven as follows. By the Trotter-Kato formula,
\begin{align*}
    |e^{-t(\LF + \kappa)} \omega|_x &= |e^{-t\LF - t \kappa} \omega|_x \\
    &= \lim_{n \rightarrow \infty} \left| \left[ (e^{-\frac{t}{n}(\tilde{\Delta})}e^{-\frac{t}{n}(R + \kappa)} \right]^n \omega_t\right|_x  \\
   &\leq e^{-t\Delta} |\omega|_x.
\end{align*}
Integrating over $M$ yields
\begin{align*}
   \| Z(\omega) \|_p^p &\leq C \left( \int_M (\sup_{t>0} |\omega_t|_x)^p dx \right)^{(2-p)/2} \left( \int_M \left[ |\omega|_x^p - \int_0^\infty \Delta |\omega_t|_x^p dt \right] dx\right)^{p/2} \\
   &\leq C \left( \int_M (\sup_{t>0} |\omega_t|_x)^p dx \right)^{(2-p)/2} \| \omega\|_p^{\frac{p^2}{2}}
\end{align*} 
Here we used Hölder's inequality with exponents $\frac{2}{2-p}$ and $\frac{2}{p}$. By \eqref{domination} we have $\sup_{t>0} |\omega_t|_x^p \leq \sup_{t>0} e^{-t\Delta} |\omega|_x^p$. Hence,
\begin{align*}
\int_M (\sup_{t>0} |\omega_t|_x)^p dx &\leq \int_M \sup_{t>0} (e^{-t\Delta} |\omega|_x)^p dx \\
&\leq C \int_M |\omega|_x^p dx
\end{align*}
because $e^{-t\Delta}$ satisfies a maximal inequality as it is a submarkovian semigroup (see \cite{steinb} p73). As a consequence one has $ \| Z(\omega) \|_p \leq C \|\omega\|_p$.
\end{proof}
As a consequence of Theorem \ref{Z} we recover the boundedness of $\mathcal{R}_\epsilon$ on $L^p$ for $ p \in [2, \infty)$. Indeed, since $Z$ is bounded on $L^{p'}$, the same techniques as in Theorem \ref{theoCoulh} gives that $\mathcal{R}_\kappa$ is bounded on $L^p$. Finally, by writing $\mathcal{R}_\epsilon = \mathcal{R}_\kappa (\Delta+\kappa)^{1/2} (\Delta+ \epsilon)^{-1/2}$ and using functional calculus for $\Delta$ we see that $\mathcal{R}_\epsilon$ is bounded on $L^p$.

\section{Horizontal LPS functions for the Hodge-de Rham Laplacian for $p\leq 2$}

In this section we prove the boundedness of $\overrightarrow{g}$ for small values of $p$ under the same assumptions as in Theorem 1. 
\begin{theorem}\label{ver}
Suppose that the negative part of Ricci curvature satisfies \eqref{soucr} for some  $\alpha \in (0,1)$ and let $p_1 := \frac{2}{1+\sqrt{1-\alpha}}$. Given $p \in (p_1,2]$ and suppose that $e^{-t\sqrt{\LF}}$ satisfies the maximal inequality \begin{equation}\label{maxineqracine}
\| \sup_{t>0} |e^{-t\sqrt{\LF}} \omega|_x \|_p \leq C \| \omega \|_p
\end{equation} 
for all $\omega \in L^p$. Then $\overrightarrow{g}$ is bounded from $L^p(\Lambda^1 T^* M)$ to  $L^p(M)$. It is also bounded on $L^q$ for $q \in [p,2]$. For $q \in [p,2]$, we have the lower estimate \begin{equation}\label{lowerestimate}
\| \overrightarrow{g}(\omega) \|_{q'} \geq C \| \omega\|_{q'}, \quad \forall \omega \in L^{q'}.
\end{equation}
\end{theorem}

\begin{proof}
Let $\omega$ be a suitable 1-form and set $\omega_t := e^{-t\sqrt{\LF}} \omega$. We compute

\begin{equation}\begin{aligned}\label{equaequa}
    \frac{\partial^2}{\partial t^2} |\omega_t|_x^p &=  p\frac{\partial}{\partial t} <-\sqrt{\LF} \omega_t, \omega_t>_x |\omega_t|_x^{p-2} \\
    &= p |\omega_t|_x^{p-2} \left[ <\LF \omega_t, \omega_t>_x + |\sqrt{\LF} \omega_t|_x^2  \right.\\
    & \hspace{3cm} + (p-2)|\omega_t|_x^{-2}<\sqrt{\LF} \omega, \omega_t>^2_x\biggr] \\
    &\geq p(p-1)|\frac{\partial}{\partial t} \omega_t|_x^2 |\omega_t|^{p-2}_x + p <\LF \omega_t, |\omega_t|_x^{p-2} \omega_t>_x.
\end{aligned}\end{equation}
Here we used the Cauchy-Schwarz inequality. Note that \eqref{equaequa} implies \begin{equation}\label{posig}
\frac{\partial^2}{\partial t^2} |\omega_t|_x^p  - p <\LF \omega_t, |\omega_t|_x^{p-2} \omega_t>_x \, \ \geq 0.
\end{equation} Set $w^* := \sup_{t>0} |w_t|_x$. We multiply \eqref{equaequa} by $t |\omega_t|^{2-p}_x$ and integrate for $t \in (0,\infty)$. Using \eqref{posig} we obtain
\begin{align*}
    \overrightarrow{g}(x)^2 &\leq C \int_0^\infty |\omega_t|_x^{2-p} \left[ \frac{\partial^2}{\partial t^2} |\omega_t|_x^p  - p <\LF \omega_t, |\omega_t|_x^{p-2} \omega_t>_x \right] t dt \\
    &\leq C (\omega^*)^{2-p}(x) \int_0^\infty  \left[ \frac{\partial^2}{\partial t^2} |\omega_t|_x^p  - p <\LF \omega_t, |\omega_t|_x^{p-2} \omega_t>_x \right]t dt  \\
    &\leq C (\omega^*)^{2-p}(x) \left[ |\omega|_x^p - p \int_0^\infty <\LF \omega_t, |\omega_t|_x^{p-2} \omega_t>_x t dt \right]. 
\end{align*}
Note that we used the fact that $t \sqrt{\LF}e^{-t\sqrt{\LF}} \omega \rightarrow 0$ in $L^2$ when $t \rightarrow +\infty$. We argue similarly as when we proved $\|e^{-t\LF} \omega \|_2 \rightarrow 0$, taking first $\omega$ in the range of $\sqrt{\LF}$ and using its density. Hölder's inequality for exponents $\frac{2}{p}$ and $\frac{2}{2-p}$ yields
\begin{equation}\label{kennylala}
    \| \overrightarrow{g} (\omega) \|_p^p \leq C \| \omega^* \|_p^{p(1-\frac{p}{2})}\left[ \int_M \left(|\omega|_x^p  - p \int_0^\infty I(t) t dt \right) dx\right]^{p/2},
\end{equation}
where $I(t) = \int_M <\LF \omega_t, \omega_t |\omega_t|_x^{p-2}>_x dx.$ The same calculations as in Proposition 11 yield $I(t) \geq 0$ for all $t > 0$. Hence, \eqref{kennylala} gives $\| \overrightarrow{g}(\omega) \|_p \leq C \| \omega\|_p$. We deduce by interpolation that $\overrightarrow{g}$ is bounded on $L^q$ for all $ q \in [p,2]$. The lower estimate is obtained as  in \cite{steinb}, p55-56.
\end{proof}

\section[LPS functions for the Schrödinger Operator in the subcritical case for $p\leq2$]{Vertical LPS functions for the Schrödinger Operator in the subcritical case for $p\leq2$}
In this section we prove Theorem \ref{theoHL}. The following lemma is obtained by applying \eqref{48} to $|f|^\frac{p}{2}$.

\begin{lemma}\label{lemmaSubcrV}
Assume that $V^-$ is subcritical with rate $\alpha$, then for all suitable $f$ in $L^p$ we have
\begin{equation}\label{SCPP}
     \int_M V^- |f|^p dx \leq \alpha \int_M \frac{p^2}{4}|\nabla f|^2 |f|^{p-2} + V^+ |f|^p dx.
\end{equation}\end{lemma}
\textbf{Proof of Theorem \ref{theoHL}.}
By the subordination formula \eqref{sub}, it is sufficient to prove the boundedness of $H_L$. We have $$ H_L(f) \leq C \left( H_L(f^+) + H_L(f^-) \right).$$ Thus it is sufficient to prove $\|H_L(f)\|_p \leq C \| f \|_p$ for all non-negative functions $f$. Let $f$ be a non-negative function and set $f_t := e^{-tL} f$. We have $f_t>0$ for all $t>0$. Let $\epsilon, \eta$ and $\xi$ be positive constants and set
\begin{equation*}
    Q(f,x,t) := (- \frac{\partial}{\partial t} -\xi \Delta - c V^+ + k V^-) f_t^p.
\end{equation*}
We have 
\begin{align*}
       Q(f,x,t) &= p [(\Delta + V) f_t] f_t^{p-1} + \xi p (p-1) |\nabla f_t|^2 f_t^{p-2}  \\ & \hspace{4cm} - \xi p (\Delta f_t) f_t^{p-1} -c V^+ f_t^p  + k V^- f_t^p \\
       &= \xi p (p-1) |\nabla f_t|^2 f_t^{p-2} + p (1- \xi) (\Delta f_t) f_t^{p-1} \\ &\hspace{4cm}+ [(p-c)V^+ + (k-p)V^- ]f_t^{p}.
\end{align*}
We multiply by $f_t^{2-p}$ to obtain
\begin{multline}
\label{50}
\xi p (p-1) |\nabla f_t|^2 +\left[(p-c) V^+  + (k - p) V^- \right]f_t^{2} = \\ f_t^{2-p} Q(f,x,t) + p(\xi -1) f_t \Delta f_t.
\end{multline}
Set  $\epsilon := p - c$ and $\eta := k-p$. If $\epsilon$ and $\eta$ are positive, the integration of \eqref{50} for $t \in [0,\infty)$ yields
\begin{equation}\label{51}
    H_L(f)^2(x) \leq C \int_0^\infty f_t^{2-p}  \left[ ( -\frac{\partial}{\partial t} - \xi \Delta  - c V^+  + kV^-)f_t^p +p (\xi -1) f_t^{p-1} \Delta f_t \right]   dt
\end{equation}
where $C$ is constant depending on $\epsilon, \eta, \xi$ and $p$. A useful consequence of \eqref{50} is \begin{equation}\label{50plus}
( -\frac{\partial}{\partial t} - \xi \Delta  - c V^+  + kV^-)f_t^p +p (\xi -1) f_t^{p-1} \Delta f_t \geq 0. \end{equation} Set $f^* := \sup_{t >0} f_t$. Using \eqref{50plus}, \eqref{51} gives
\begin{multline}\label{666}
   H_L(f)^2(x)  \leq \\ C (f^*)^{2-p} \int_0^\infty \left[ ( -\frac{\partial}{\partial t} - \xi \Delta  - c V^+  + kV^- )f_t^p +p (\xi -1) f_t^{p-1} \Delta f_t \right] dt.
\end{multline}
Set \begin{equation*}
    I(t) := \int_M (- c V^+  + kV^-) f_t^p - p (1- \xi )f_t^{p-1} \Delta f_t.
\end{equation*}
By Hölder's inequality, \eqref{666} implies
\begin{align*}
    \| H_L(f)\|^p_p &\leq C \left[ \int_M (f^*)^{p} dx \right]^{\frac{2-p}{2}} \times \\ & \quad \left[ \int_M \int_0^\infty  \left[ ( -\frac{\partial}{\partial t} - \xi \Delta  - c V^+  + kV^- )f_t^p +p (\xi -1) f_t^{p-1} \Delta f_t \right] \, dx dt\right]^{\frac{p}{2}} \\
    & \leq C \|f^*\|_p^{\frac{p(2-p)}{2}}\left[ \|f\|_p^p + \int_0^\infty I(t) dt \right]^{\frac{p}{2}}.
\end{align*}
By Lemma \ref{lemmaIpourL} below, we can choose $c, k$ and $\eta$ such that $I(t)  \leq 0$ for all $t > 0$. Hence,
\begin{equation*}
     \| H_L(f)\|^p_p \leq C \|f^*\|_p^{\frac{p(2-p)}{2}} \| f\|_p^{\frac{p^2}{2}}.
\end{equation*}
The same argument as is Proposition \ref{propDecreasing} implies that $e^{-tL}$ is a contraction semigroup on $L^p$ for all $p \in (p_1, {p_1}')$. It is also a classical fact that $e^{-tL}$ is a positive semigroup. For a positive contraction and analytic semigroup one has $\| f^* \|_p \leq C \|f \|_p$ (see \cite{lemerdy}, Corollary 4.1). Hence,
\begin{equation*}
     \| H_L(f)\|^p_p \leq C \|f \|_p^p.
\end{equation*}
\begin{lemma}\label{lemmaIpourL}
Under the subcriticality assumption, for all $p \in (p_1,2]$ there exist $c$, $k$ and $\xi$ positive constants such that $I(t) \leq 0$ for all $t>0$.
\end{lemma}
\begin{proof}
By Lemma \ref{lemmaSubcrV} and integration by parts,
\begin{align*}
    I(t) &\leq \int_M (\alpha  k -  c) V^+ f^p + \alpha \frac{p^2}{4} |\nabla f_t|^2 f_t^{p-2} - p(1- \xi) f_t^{p-1} \Delta f_t \, dx  \\
    & \leq \int_M (\alpha  k -  c) V^+ f^p + \left[\alpha k \frac{p^2}{4}-p(p-1)(1-\xi)\right]|\nabla f_t|^2 f_t^{p-2} \, dx.
\end{align*}
Choose $k,c$ and $\eta$ such that
\begin{equation*}
  \left\{
      \begin{aligned}
   p \alpha k&\leq 4 (p-1)(1-\xi) \\
      \alpha k &\leq c.
 \end{aligned} \right.
 \end{equation*}
The same discussion as in Lemma \ref{lemmaIsubc} gives that $\xi, \eta$ and $\epsilon$ can be chosen to allow any value of $p \in (p_1,2]$. \end{proof}
 
 \section[LPS functions for the Schrödinger Operator in the subcritical case for $p>2$.]{Vertical LPS functions for the Schrödinger Operator in the subcritical case for $p>2$}
 
 In this section, we assume the manifold has the doubling property, that is there exists a positive constant $C$ such that for all $x \in M$ and $r > 0$, \begin{equation} \label{doubling}
 Vol(x,2r) \leq C Vol(x,r)
 \end{equation} where $Vol(x,r)$ is the volume of the ball of center $x$ and radius $r$ for the Riemannian distance $\rho$. This is equivalent to the fact that for some constants $C$ and $N$,
 \begin{equation} \label{N}
 Vol(x,\lambda r) \leq C \lambda^N Vol(x,r)
 \end{equation} for all $\lambda \geq 1$. We suppose in addition that the heat kernel $p_t(x,y)$ associated with $\Delta$ has a Gaussian upper estimate, that is there exist positive constants $C$ and $c$ such that \begin{equation}\label{gaussian}
 p_t (x,y) \leq C \frac{e^{-c\rho^2(x,y)/t}}{Vol(x,\sqrt{t})}.
 \end{equation} 
 Under these assumptions, the semigroup $e^{-tL}$ is uniformly bounded on $L^p(M)$ for all $p \in (p_0,{p_0}')$ where ${p_0}' := \frac{2}{1-\sqrt{1-\alpha}} \frac{N}{N-2}$. 
 If $N \leq 2$ it is true for ${p_0}' := +\infty$. 
 Under some integrability conditions on $V$, it is proven in \cite{AssaadOuhabaz} (Theorem 3.9) that the Riesz transform $dL^{-1/2}$ is bounded on $L^p$ for $p$ in some interval $[2,q)$ if $d\Delta^{-1/2}$ is also bounded.
 We recall that $H_L$ is defined by $$H_L(f)(x) = \left( \int_0^\infty |\nabla e^{-tL}f|^2(x) + |V| (e^{-tL}f)|^2(x) \, dt\right)^{1/2}.$$ Note that if $V=0$ we obtain $H_\Delta = H$. We prove, under similar integrability contiditions on $V$, that the Littlewood-Paley-Stein function $H_L$ is bounded in the same interval. We recall a proposition from \cite{AssaadOuhabaz}.
 
 \begin{proposition}\label{Assaad}
 Assume that $V^-$ satisfies \eqref{48} for some $\alpha \in (0,1)$. Let ${p_0}' = + \infty $ if $N \leq 2$ and ${p_0}' = \frac{2}{1-\sqrt{1-\alpha}} \frac{N}{N-2}$ if $N > 2$.  
 Given $p$ and $q$ such that $p_0< p \leq q < {p_0}'$ and set $\frac{1}{r} := \frac{1}{p} - \frac{1}{q}$, then the family of operators $Vol(x,\sqrt{t})^{\frac{1}{r}}e^{-tL}$ is uniformly bounded from $L^p(M)$ to $L^q(M)$. 
 By duality, the family $e^{-tL} Vol(x,\sqrt{t})^{\frac{1}{r}}$ is uniformly bounded from $L^{q'}(M)$ to $L^{p'}(M)$ for all ${p_0} < q' \leq p' < {p_0}'$ with $\frac{1}{r} = \frac{1}{q'} - \frac{1}{p'}$.
 \end{proposition}
Set \begin{equation*} \left\{\begin{aligned}
H_L^{(\nabla)}(f) &:= \left(\int_0^\infty |\nabla e^{-tL} f|^2 dt \right)^{1/2}, \\
H_L^{(V)}(f) &:= \left(\int_0^\infty |V| |e^{-tL} f|^2 dt \right)^{1/2}.
\end{aligned} \right. \end{equation*}
Note that $H_L(f) \leq \sqrt{2} \left[H_L^{(\nabla)}(f) + H_L^{(V)}(f) \right].$  We will use the next proposition which follows from Proposition 4.2 in \cite{riesz}.

 \begin{proposition}\label{lemmagradsemi}
 If $H^{(\nabla)}_L$ is bounded of $L^p$, then there exists a positive constant $C$ such that for all $f \in L^p$ and $t>0$, \begin{equation}
 \label{gradsemi} \| \nabla e^{-tL} f \|_p \leq \frac{C}{\sqrt{t}} \| f\|_p.
 \end{equation}
 \end{proposition}
The main result of this section is the following.
\begin{theorem} \label{psup}
Assume that $V^-$ satisfies \eqref{48} for some $\alpha \in (0,1)$. Define $p_0$ as before. Suppose there exist $r_1, r_2 > 2$ such that
\begin{equation}\label{hypotheseintegre}
\int_0^1 \| \frac{V}{Vol(.,\sqrt{t})^{\frac{1}{r_1}}}\|^2_{r_1} t dt < \infty, \,   \int_1^\infty \| \frac{V}{Vol(.,\sqrt{t})^{\frac{1}{r_2}}}\|^2_{r_2} t dt < \infty .
\end{equation}
Set $r := \inf(r_1,r_2)$. If $N > 2$, let $p \in [2,\frac{{p_0}'r}{{p_0}'+r})$ and assume that $H_\Delta$ is bounded on $L^p$, then $H_L^{(\nabla)}$ is bounded on $L^p$. If $N \leq 2$, let $p \geq 2$ and assume $H_\Delta$ is bounded on $L^p$, then $H_L^{(\nabla)}$ is bounded on $L^p$.
\end{theorem}

\begin{proof}

Duhamel's formula for semigroups says that for all $f \in L^p(M)$ 
\begin{equation}\label{Duhamel}
e^{-tL} f = e^{-t\Delta} f - \int_0^t e^{-s\Delta} V e^{-(t-s)L} f ds.
\end{equation}
It follows that \begin{equation*}
|\nabla e^{-tL}f|^2 \leq 2 \left[ |\nabla e^{-t\Delta}f|^2  + \left|\int_0^t \nabla e^{-s\Delta} V e^{-(t-s)L} f ds\right|^2 \right]. \end{equation*} After integration on $(0, \infty)$ we obtain 
\begin{equation*}
(H_L^{(\nabla)}(f))^2 \leq C \left[ (H_\Delta(f))^2 + \int_0^\infty \left|\int_0^t \nabla e^{-s\Delta} V e^{-(t-s)L} f ds\right|^2 dt \right].
\end{equation*}
Therefore,
\begin{align*}
\|H_L^{(\nabla)}(f)\|^2_p \leq C  \left[ \|H_\Delta(f)\|^2_p + \left\| \int_0^\infty \left|\int_0^t \nabla e^{-s\Delta} V e^{-(t-s)L} f ds\right|^2 dt\right\|_{p/2} \right] \\
\leq C  \left[ \|f\|^2_p + \left\| \int_0^\infty \left|\int_0^t \nabla e^{-s\Delta} V e^{-(t-s)L} f ds\right|^2 dt\right\|_{p/2} \right].
\end{align*}
Here we used the boundedness of $H_\Delta$ on $L^p$. 
Hence, it is sufficient to establish that \begin{equation}
\label{finalestimate}  \left\| \int_0^\infty \left|\int_0^t \nabla e^{-sL} V e^{-(t-s)\Delta} f ds\right|^2 dt \right\|_{p/2} \leq C \| f \|_p^2.
\end{equation} We have
\begin{align*}
\Biggl\| \int_0^\infty \left|\int_0^t \nabla e^{-s\Delta} V e^{-(t-s)L} f ds\right|^2 dt\Biggr{\|}_{p/2} &\leq \int_0^\infty \left\| \left| \int_0^t  \nabla e^{-\Delta} V e^{-(t-s)L} fds \right|^2  \right\|_{p/2} dt \\
&\leq \int_0^\infty \left( \int_0^t   \left\| \nabla e^{-s\Delta} V e^{-(t-s)L} f \right\|_p ds \right)^2 dt.
\end{align*}
Set 
\begin{equation*}\left\{ \begin{aligned}
I_1 &:= \int_0^1  \left( \int_{0}^{t/2} \left\| \nabla e^{-s\Delta} V e^{-(t-s)L} f \right \|_p ds \right)^2 dt,\\
I_2 &:= \int_1^\infty  \left( \int_{0}^{t/2} \left\|  \nabla e^{-s\Delta} V e^{-(t-s)L} f \right \|_p ds \right)^2 dt,\\
I_3 &:= \int_0^1 \left( \int_{t/2}^{t} \left\|  \nabla e^{-s\Delta} V e^{-(t-s)L} f ds \right \|_p ds \right)^2 dt,\\
I_4 &:= \int_1^\infty \left(  \int_{t/2}^t \left\|  \nabla e^{-s\Delta} V e^{-(t-s)L} f\right \|_p ds \right)^2 dt.  \end{aligned} \right. \end{equation*}
We have $$ \left\| \int_0^\infty \left|\int_0^t \nabla e^{-s\Delta} V e^{-(t-s)L} f ds\right|^2 dt \right \|_{p/2} \leq 2 \left[ I_1 + I_2 + I_3 + I_4 \right].$$ 
We prove that each term $I_1, I_2, I_3$ and $I_4$ is bounded by $ C \|f \|_p^2$. By Proposition \ref{lemmagradsemi},
\begin{align*}
I_1 &\leq C \int_0^1  \left| \int_{0}^{t/2} s^{-1/2} \left\|  V e^{-(t-s)L} f ds \right\|_p ds \right|^2 dt  \\
&\leq C\int_0^1  \left| \int_{0}^{t/2} s^{-1/2} \left\|  \frac{V}{Vol(x,\sqrt{t-s})^\frac{1}{r_1}} Vol(x,\sqrt{t-s})^\frac{1}{r_1}e^{-(t-s)L} f \right\|_p ds \right|^2 dt  \\
&\leq C \int_0^1  \left| \int_{0}^{t/2} s^{-1/2} \left\|  \frac{V}{Vol(x,\sqrt{t-s})^\frac{1}{r_1}} \right\|_{r_1} \left\|Vol(x,\sqrt{t-s})^\frac{1}{r_1}e^{-(t-s)L} f \right\|_{q_1} ds \right|^2 dt \end{align*}
where $\frac{1}{p} = \frac{1}{r_1} + \frac{1}{q_1}$. Note that here $q_1$ has to satisfy $q_1 < {p_0}'$ which gives $p < \frac{{p_0}' r_1}{{p_0}' + r_1}$. Since $s < t/2$, $Vol(x,\sqrt{t-s}) \geq Vol(x,\sqrt{t/2})$. Thus,  
\begin{align*}
I_1 &\leq C \int_0^1  \left| \int_{0}^{t/2} s^{-1/2} \left\|  \frac{V}{Vol(x,\sqrt{t/2})^{\frac{1}{r_1}}} \right\|_{r_1} \left\|Vol(x,\sqrt{t-s})^{\frac{1}{r_1}}e^{-(t-s)L} f  \right\|_{q_1} ds \right|^2 dt  \\
&\leq C \int_0^1  \left\|  \frac{V}{Vol(x,\sqrt{t/2})^{\frac{1}{r_1}}} \right\|^2_{r_1} \left| \int_{0}^{t/2} s^{-1/2} \left\| Vol(x,\sqrt{t-s})^{\frac{1}{r_1}}e^{-(t-s)L} f \right\|_{q_1} ds \right|^2 dt.
\end{align*}
By Proposition \ref{Assaad}, $\|Vol(x,\sqrt{t-s})^{\frac{1}{r_1}}e^{-(t-s)L} f \|_{q_1} \leq C \|f\|_p$. Thus,
\begin{equation}\label{I1}
I_1 \leq  C \left( \int_0^{1/2}   \left\|  \frac{V}{Vol(x,\sqrt{t})^{\frac{1}{r_1}}} \right\|^2_{r_1} t dt\right) \|f \|_p^2 = C' \|f \|_p^2.
\end{equation}
We prove as for $I_1$ that \begin{equation}\label{I2}
I_2 \leq  C \left( \int_{1/2}^{\infty}   \left\|  \frac{V}{Vol(x,\sqrt{t})^{\frac{1}{r_2}}} \right\|^2_{r_2} t dt\right) \|f\|_p^2 = C' \|f\|_p^2.
\end{equation}
Note that reproducing the previous proof for $I_2$ implies that we have to choose $p < \frac{{p_0}'r_2}{{p_0}' + r_2}.$ Now we bound $I_3$ and $I_4$. By Proposition \ref{lemmagradsemi} we have
\begin{align*}
I_3 &= \int_0^1 \left| \int_{t/2}^t \left\| \nabla  e^{-\frac{s}{2}\Delta} e^{-\frac{s}{2}\Delta} V e^{-(t-s)L} f ds \right\|_p ds \right|^2 dt \\
&\leq C \int_0^1 \left| \int_{t/2}^t s^{-1/2} \left\|    e^{-\frac{s}{2}\Delta} V e^{-(t-s)L} f ds\right\|_p ds \right|^2 dt \\
&= C \int_0^1 \left| \int_{t/2}^t s^{-1/2} \left\|   e^{-\frac{s}{2}\Delta} \frac{Vol(x,\sqrt{s/2})^\frac{1}{r_1}}{Vol(x,\sqrt{s/2})^\frac{1}{r_1}} V e^{-(t-s)L} f ds\right\|_p ds \right|^22 dt. 
\end{align*}
By Proposition \ref{Assaad},  $e^{-\frac{s}{2}\Delta} Vol(x,\sqrt{\frac{s}{2}})$ is bounded from $L^{q_1}(M)$ to $L^{p}(M)$ with $\frac{1}{p} = \frac{1}{q_1} - \frac{1}{r_1}$.
Thus,

\begin{align*}
I_3 &\leq C \int_0^1 \left| \int_{t/2}^t s^{-1/2} \left\|   \frac{V}{Vol(x,\sqrt{s/2})^\frac{1}{r_1}} e^{-(t-s)L} f \right\|_{q_1} ds \right|^2 dt \\
&\leq C \int_0^1 \left| \int_{t/2}^t s^{-1/2} \left\|   \frac{V}{Vol(x,\sqrt{s/2})^\frac{1}{r_1}}\right\|_{r_1} \left\| e^{-(t-s)L} f \right\|_{p} ds \right|^2 dt  \\
&\leq C \int_0^1 \left| \int_{t/2}^t s^{-1/2} \left\|   \frac{V}{Vol(x,\sqrt{t/4})^\frac{1}{r_1}}\right\|_{r_1} \left\| e^{-(t-s)L} f \right\|_{p} ds \right|^2 dt
\end{align*} because $\frac{s}{2} \geq \frac{t}{4}$.  By the uniform boundedness of  $e^{-(t-s)L}$ on $L^p$ we have 
\begin{multline}\begin{aligned}\label{I3}
I_3 &\leq C \left( \int_0^1 \left\|   \frac{V}{Vol(x,\sqrt{t/4})^\frac{1}{r_1}}\right\|^2_{r_1}  \left[ \int_{t/2}^t s^{-1/2}  ds \right]^2 dt \right) \|f\|^2_p \\
&\leq C \left( \int_{0}^{1}   \left\|  \frac{V}{Vol(x,\sqrt{t/4})^{\frac{1}{r_1}}}
\right\|^2_{r_1} t dt\right) \|f\|^2_p \\
&= C' \|f\|^2_p.
\end{aligned}
\end{multline}
We prove in a similar way that \begin{equation}\label{I4}
 I_4 \leq  C \left( \int_{1/4}^{\infty}   \left\|  \frac{V}{Vol(x,\sqrt{t})^{\frac{1}{r_2}}} \right\|^2_{r_2} t dt \right) \|f\|^2_p = C' \|f\|_p^2.
\end{equation}
Combining \eqref{I1}, \eqref{I2}, \eqref{I3} and \eqref{I4} with \eqref{hypotheseintegre} we obtain \eqref{finalestimate}. Hence $\|H_L^{(\nabla)}(f)\|_p \leq C \|f \|_p$. \end{proof}
Finally we have a similar result for $H_L^{(V)}$.
\begin{theorem}\label{psup2}
Let $p>2$, assume $V^-$ satisfies \eqref{48} for some $\alpha \in (0,1)$. Define $p_0$ as before. Suppose there exist $r_1, r_2 > 2$ such that
\begin{equation}\label{hypotheseintegre2}
\int_0^1 \| \frac{|V|^{1/2} }{Vol(.,\sqrt{t})^{\frac{1}{r_1}}}\|^2_{r_1} dt < \infty  ,  \int_1^\infty \| \frac{|V|^{1/2}}{Vol(.,\sqrt{t})^{\frac{1}{r_2}}}\|^2_{r_2}  dt < \infty .
\end{equation}
If $N > 2$, then $H_L^{(V)}$ is bounded on $L^p(M)$ for $p \in [2,\frac{{p_0}'r}{{p_0}'+r})$, where $r = \inf(r_1,r_2)$. If $N \leq 2$, then $H_L^{(V)}$ is bounded on $L^p(M)$ for $p \in [2,\infty)$.
\end{theorem}
\begin{proof} We have
\begin{align*}
\|H_L^{(V)} (f) \|_p &= \biggl\| \int_0^\infty \left| |V|^{1/2} e^{-tL}f\right|^{1/2} dt \biggr{\|}^{1/2}_{p/2} \\
& \leq \left( \int_0^\infty \biggl\| |V|^{1/2} e^{-tL} f\biggr{\|}^2_p dt \right)^{1/2} \\
& \leq \sqrt{2} \left[ \left( \int_0^1 \biggl\| |V|^{1/2} e^{-tL} f\biggr{\|}^2_p dt \right)^{1/2} + \left( \int_1^\infty \biggl\| |V|^{1/2} e^{-tL} f\biggr{\|}^2_p dt \right)^{1/2} \right].
\end{align*}
We bound the two latter integrals separately. One has
\begin{align*}
\int_0^1 \biggl\| |V|^{1/2} e^{-tL} f\biggr{\|}^2_p dt &=  \int_0^1 \biggl\| \frac{|V|^{1/2}}{Vol(x,\sqrt{t})^{1/r_1}}Vol(x,\sqrt{t})^{1/r_1} e^{-tL} f\biggr{\|}^2_p dt \\
&\leq \int_0^1 \biggl\| \frac{|V|^{1/2}}{Vol(x,\sqrt{t})^{1/r_1}}\biggr{\|}^2_{r_1} \biggl\| Vol(x,\sqrt{t})^{1/r_1} e^{-tL} f\biggr{\|}^2_{q_1} dt
\end{align*} where $\frac{1}{p} = \frac{1}{r_1} + \frac{1}{q_1}$. By Proposition \ref{Assaad}, we have $\| Vol(x,\sqrt{t})^{1/r_1} e^{-tL} f \|_{q_1} \leq C \|f\|_p$. Hence,
\begin{equation}\label{2022}
 \int_0^1 \biggl\| |V|^{1/2} e^{-tL} f\biggr{\|}^2_p dt \leq C  \left(\int_0^1 \biggl\| \frac{|V|^{1/2}}{Vol(x,\sqrt{t})^{1/r_1}}\biggr{\|}^2_{r_1} dt \right) \| f \|^2_p.
\end{equation}
Note that here $q_1$ has to satisfy $q_1 < {p_0}'$ what gives $p < \frac{{p_0}' r_1}{{p_0}' + r_1}$. The same argument gives 
\begin{equation}\label{2023}
 \int_1^\infty \biggl\| |V|^{1/2} e^{-tL} f\biggr{\|}^2_p dt \leq C \left(\int_1^\infty \biggl\| \frac{|V|^{1/2}}{Vol(x,\sqrt{t})^{1/r_2}}\biggr{\|}^2_{r_2} dt \right) \| f \|^2_p.
\end{equation}
Here we need $p < \frac{{p_0}' r_2}{{p_0}' + r_2}$. Together with \eqref{hypotheseintegre2}, \eqref{2022} and \eqref{2023} yield 
\begin{equation*}
\| H_L^{(V)} (f) \|_p \leq C \| f \|_p.
\end{equation*}
\end{proof}
\textbf{Remarks.}
If $Vol(x,t)$ has polynomial growth $Vol(x,t) \simeq t^N$ (for example, in $\mathbb{R}^N$) then the conditions \eqref{hypotheseintegre} or \eqref{hypotheseintegre2} read as $V \in  L^{\frac{N}{2}- \epsilon} \cap L^{\frac{N}{2} + \epsilon}$ for some positive $\epsilon$. In the general setting, we could not find implications between \eqref{hypotheseintegre} and \eqref{hypotheseintegre2}. In \cite{ouh}, it is shown that if $V \geq 0$ is not identically zero, then $H_L$ is not bounded on $L^p(\mathbb{R}^N)$ for $p>N$ if we assume there exists a positive bounded function $\phi$ such that $e^{-tL} \phi = \phi$ for all $t \geq 0$. It is true for a wide class of potentials, for example if $V \in L^{\frac{N}{2}- \epsilon} \cap L^{\frac{N}{2} + \epsilon}$ (see \cite{mcgil}). For a discussion on this property, see \cite{murata}. 

\section*{Acknowledgments}I would like to thank my PhD supervisor El Maati Ouhabaz for his help and advice during the preparation of this article. This research is partially supported by the ANR project RAGE "Analyse Réelle et Géométrie" (ANR-18-CE40-0012).
\bibliography{bibli}
\bibliographystyle{plain}
\end{document}